\documentclass{amsart}
\usepackage{amsfonts,amssymb,amscd,amsmath,enumerate,verbatim,calc}
\usepackage[all]{xy}

\newcommand{\CM}{Cohen-Macaulay}

\newcommand{\wrt}{with respect to}

\newcommand{\n}{\mathfrak{n} }
\newcommand{\m}{\mathfrak{m} }

\newcommand{\q}{\mathfrak{q} }

\newcommand{\K}{\mathbf{K} }

\newcommand{\Db}{\mathbf{D} }
\newcommand{\Cb}{\mathbf{C} }
\newcommand{\Pb}{\mathbf{P} }
\newcommand{\Qb}{\mathbf{Q} }

\newcommand{\rt}{\rightarrow}
\newcommand{\xar}{\longrightarrow}
\newcommand{\ov}{\overline}

\newcommand{\bx}{\mathbf{x}}

\newcommand{\by}{\mathbf{y}}
\newcommand{\bff}{\mathbf{f} }

\newcommand{\wt}{\widetilde }

\newcommand{\image}{\operatorname{image}}

\newcommand{\codim}{\operatorname{codim}}

\newcommand{\Tr}{\operatorname{Tr}}
\newcommand{\Om}{\Omega }
\newcommand{\ann}{\operatorname{ann}}

\newcommand{\cx}{\operatorname{cx}}
\newcommand{\cone}{\operatorname{cone}}
\newcommand{\rad}{\operatorname{rad}}

\newcommand{\End}{\operatorname{End}}
\newcommand{\Endu}{\operatorname{\underline{End}}}
\newcommand{\projdim}{\operatorname{projdim}}
\newcommand{\injdim}{\operatorname{injdim}}

\newcommand{\Hom}{\operatorname{Hom}}
\newcommand{\Ext}{\operatorname{Ext}}
\newcommand{\Tor}{\operatorname{Tor}}
\newcommand{\CMg}{\operatorname{CM^g}}
\newcommand{\CMr}{\operatorname{CM^r}}
\newcommand{\CMS}{\operatorname{\underline{CM}}}
\theoremstyle{plain}

\newtheorem{theorem}{Theorem}[section]
\newtheorem{corollary}[theorem]{Corollary}
\newtheorem{lemma}[theorem]{Lemma}
\newtheorem{proposition}[theorem]{Proposition}

\theoremstyle{definition}
\newtheorem{definition}[theorem]{Definition}
\newtheorem{s}[theorem]{}
\newtheorem{remark}[theorem]{Remark}

\theoremstyle{remark}

\begin{document}

\title{Invariants of  Linkage of modules}
 \author{Tony J. Puthenpurakal}
\date{\today}
\address{Department of Mathematics, Indian Institute of Technology Bombay, Powai, Mumbai 400 076, India}
\email{tputhen@math.iitb.ac.in}
\begin{abstract}
Let $(A,\m)$ be a Gorenstein local ring and let $M, N$ be two \CM \ $A$-modules 
with $M$ linked to $N$ via a Gorenstein ideal $\q$.
  Let $L$ be another finitely generated $A$-module. We show that $\Ext^i_A(L,M) = 0 $ for all $i \gg 0$ if and only if $\Tor^A_i(L,N) = 0$ for all $i \gg 0$. If $D$ is \CM \ then we show that $\Ext^i_A(M, D) = 0 $ for all $i \gg 0$ if and only if 
  $\Ext^i_A(D^\dagger, N) = 0$ for all $i \gg 0$,  where $D^\dagger = \Ext^r_A(D,A)$ and $r = \codim D$.
  As a consequence we get that $\Ext^i_A(M, M) = 0 $ for all $i \gg 0$ if and only if 
  $\Ext^i_A(N, N) = 0$ for all $i \gg 0$.
   We also show that $\End_A(M)/\rad \End_A(M) \cong (\End_A(N)/\rad \End_A(N))^{op}$. We also give a negative answer to a question of  Martsinkovsky and Strooker.

\end{abstract}

\maketitle

\section{introduction}
Let $(A,\m)$ be a Gorenstein local ring. Recall an ideal $\q$ in $A$ is said to be a \textit{Gorenstein ideal} if $\q$ is perfect and $A/\q$ is a Gorenstein local ring. A special class of Gorenstein ideals are \textit{CI(complete intersection) ideals}, i.e., ideals generated by an $A$-regular sequence. In this paper, a not necessarily perfect ideal $\q$ such that $A/\q$ is a Gorenstein ring will be called a \textit{quasi-Gorenstein }ideal. 

Two ideals $I$ and $J$ are linked by a Gorenstein ideal $\q$ if $\q \subseteq I \cap J$;  $ J = (\q \colon I)$ and $I = (\q \colon J)$. We write it as $I \sim_\q J$. If $\q$ is a complete intersection then we say $I$ is CI-linked to $J$ via $\q$. If $\q$ is a quasi-Gorenstein ideal then we say $I$ is quasi-linked to $J$ via $\q$.
Note that traditionally only CI-linkage used to be considered. However in recent times more general types of linkage are studied.

We say ideals $I$ and $J$ is in the \emph{same linkage class} if there is a sequence of ideals $I_0,\ldots, I_n$ in $A$ and Gorenstein ideals $\q_0 \ldots,\q_{n-1}$ such that
\begin{enumerate}[\rm (i)]
\item
$I_j \sim_{\q_j} I_{j+1}$, for $j = 0,\ldots, n-1$.
\item
 $I_0 = I$ and $I_n =J$. 
\end{enumerate}
If $n$ is even then we say that $I$ and $J$ are \emph{evenly linked}.
We can analogously define CI-linkage class, quasi-linkage class, even CI-linkage class and even quasi-linkage class (of ideals).

A natural question is that if $I$ and $J$ are in the same linkage class then what properties of $I$ is shared by $J$. This was classically done when $I$ and $J$ are in the same CI-linkage class (or even CI-linkage class). In their landmark paper \cite{PS}, Peskine and Szpiro proved that if $I$ and $J$ are in the same CI-linkage class and  $I$ is a Cohen-Macaulay ideal (i.e., $A/I$ is a Cohen-Macaulay  ring) then so is $J$. This can be proved more generally for ideals in a quasi-linkage class, see \cite[Corollary 15, p.\ 616]{MS}. In another landmark paper \cite[1.14]{H}, Huneke proved that if $I$ is in the CI-linkage class of a complete intersection then the Koszul homology $H_i(I)$ are \CM \ for all $i \geq 0$. It is known that this result is not
true in if $I$ is linked to a complete intersection ( via Gorenstein ideals and not-necessarily CI-ideals). If $A = k[[X_1,\ldots, X_n]]$ (where $k$ is a field or a complete discrete valuation ring) then Huneke defined some invariants of even CI-linkage class of equidimensional unmixed ideals, see \cite[3.2]{H2}. Again these invariants are not stable under Gorenstein (even)-liason.

In a remarkable paper Martsinkovsky and Strooker, \cite{MS},  introduced liason for modules. See section two for definition. We note here that ideals $I$ and $J$ are linked as ideals if and only if $A/I$ is linked to $A/J$ as modules. One can analogously define linkage class of modules, even linkage of modules etc. We can also define CI linkage of modules, quasi-linkage of modules etc. 

Thus a natural question arises: If $M, N$ are in the same linkage class of modules (or same even linkage class of modules) then what properties of $M$ are shared by $N$. The generalization of Peskine and Szpiro's result holds. If $M$ is \CM \ and $N$ is quasi-linked to $M$ then $N$ is also \CM, see \cite[Corollary 15, p.\ 616]{MS}.  
To state another property which is preserved under linkage first 
let us  recall the definition of Cohen-Macaulay approximation from \cite{AB}. A  \CM \ approximation of a
  finitely generated $A$-module $M$ is a exact sequence
\[
0 \xar Y \xar X \xar M \xar 0
\]
where $X$ is a maximal \CM \ $A$-module and $Y$ has finite projective dimension.
Such a sequence is not unique but $X$ is known to unique up to a free summand and so is well defined in the stable category   $\CMS(A)$ of maximal \CM \ $A$-modules. We denote by $X(M)$ the maximal \CM \ approximation of $M$.
In \cite[Theorem 13, p.\ 620]{MS}, Martsinkovsky and Strooker proved that if $M$ is evenly  linked to $N$ then  $X(M) \cong X(N)$ in $\CMS(A)$. They also asked if 
this result holds for $M$ and $N$ are in the same even quasi-linkage class, see
 \cite[Question 3, p.\ 623]{MS}. A motivation for this paper was to try and solve this question.
We answer this question in the negative. We prove
\begin{theorem}
\label{second}
There exists a complete intersection  $A$ of dimension one and finite length modules $M, N$ such that $M$ is evenly quasi-linked to $N$ but $X(M) \ncong X(N)$.
\end{theorem}
 To prove our next result we introduce a construction (essentially due to Ferrand) which is very useful, see section three.  Let $\CMg(A)$ be the full subcategory of \CM \ $A$-modules of codimension $g$. We prove:
\begin{theorem}
\label{third}
Let $(A,\m)$ be a Gorenstein local ring of dimension $d$. Let $M\in \CMg(A)$.
Assume $M \sim_\q N$ where
  $\q$ is  a Gorenstien ideal in $A$.
  Let $L$ be 
a finitely generated $A$-module and let $D \in \CMr(A)$. Set $D^\dagger = \Ext^r_A(D, A)$. Then
\begin{enumerate}[\rm (1)]
\item
$\Ext^i_A(L, M)  = 0$ for all $i \gg 0$ if and only if $\Tor^A_i(L, N) = 0$ for all $i \gg 0$.
\item
$\Ext^i_A(M, D)  = 0$ for all $i \gg 0$ if and only if $\Ext^i_A(D^\dagger, N)  = 0$ for all $i \gg 0$.
\end{enumerate}
\end{theorem}
We should note that this result is new even in the case for cyclic modules. 
A remarkable consequence of Theorem \ref{third} is the following result
\begin{corollary}\label{third-cor}
Let $(A,\m)$ be a Gorenstein local ring of dimension $d$. Let $M\in \CMg(A)$ and $C \in \CMr(A)$. Assume $M \sim_\q N$ and $C \sim_\n D$ where $\q, \n$ are Gorenstein ideals in $A$. Then
\[
\Ext^i_A(M, C)  = 0 \  \text{for all $i \gg 0$} \iff \Ext^i_A(D, N) = 0 \  \text{for all $i \gg 0$}.
\]
In particular
\[
\Ext^i_A(M, M)  = 0 \  \text{for all $i \gg 0$} \iff \Ext^i_A(N, N) = 0 \  \text{for all $i \gg 0$}.
\]
\end{corollary}

 We also prove the following surprising invariant of quasi-linkage.
\begin{theorem}\label{first}
Let $(A,\m)$ be a Gorenstein local ring and let $M,N,N^\prime \in \CMg(A)$.
Assume $M$ is quasi-evenly linked to $N$ and that it is quasi-oddly linked to $N^\prime$. Then
\begin{enumerate}[\rm (1)]
\item
$\End(M)/\rad \End(M) \cong \End(N)/\rad \End(N)$.
\item
$\End(M)/\rad \End(M) \cong \left(\End(N^\prime)/\rad \End(N^\prime) \right)^{op}$.
\end{enumerate}
\end{theorem}
Here if $\Gamma$ is a  ring then $\Gamma^{op}$ is its opposite ring.

We now describe in brief the contents of this paper. In section two we recall some preliminaries regarding linkage of modules as given in \cite{MS}. In section three we give a construction which is needed to prove our resuts. We prove Theorem \ref{third}(1) in section four and Theorem \ref{third}(2) in section five. We recall some facts regarding cohomological operators in section six. This is needed in section seven where we prove Theorem \ref{second}. Finally in section eight we prove Theorem \ref{first}.
\section{Some preliminaries on Liason of Modules }
In this section we recall the definition of linkage of modules as given in \cite{MS}.  Throughout 
$(A,\m)$ is  a Gorenstein local ring of dimension $d$.

\s Let us recall the definition of transpose of a module. Let $F_1 \xrightarrow{\phi} F_0 \rt M \rt 0$ be a minimal presentation of $M$.  Let $(-)^* = \Hom(-,A)$. The \textit{transpose} $\Tr(M)$ is defined by the exact sequence
\[
0 \rt M^* \rt F_0^* \xrightarrow{\phi^*} F_1^* \rt \Tr(M) \rt 0.
\] 
Also let $\Om(M)$ be the first syzygy of $M$.
\begin{definition}
Two $A$-modules $M$ and $N$ are said to be \textit{horizontally linked} if 
$M \cong \Om(\Tr(N))$ and $N \cong \Om(\Tr(M))$.
\end{definition}
Next we define linkage in general.
\begin{definition}
Two $A$-modules $M$ and $N$ are said to be linked via a Gorenstein ideal $\q$ if
\begin{enumerate}
\item
$\q \subseteq \ann M \cap \ann N$, and
\item
$M$ and $N$ are horizontally linked as $A/\q$-modules.
\end{enumerate}
We write it as $M \sim_\q N$.
\end{definition}
If $\q$ is a complete intersection we say $M$ is CI-linked to $N$ via $\q$. If $\q$ is a quasi Gorenstein ideal then we say $M$ is quasi-linked to $N$ via $\q$.

\begin{remark}
 It can be shown that ideals $I$ and $J$ are linked by a quasi-Gorenstein ideal $\q$ (definition as in the introduction) if and only if the  module $A/I$ is quasi-linked to $A/J$ by $\q$, see \cite[Proposition 1, p.\ 592]{MS}. 
\end{remark}

\s We say $M, N$ are in  \emph{same linkage class} of modules  if there is a sequence of $A$-modules $M_0,\ldots, M_n$  and Gorenstein ideals $\q_0 \ldots,\q_{n-1}$ such that
\begin{enumerate}[\rm (i)]
\item
$M_j \sim_{\q_j} M_{j+1}$, for $j = 0,\ldots, n-1$.
\item
 $M_0 = M$ and $M_n =N$. 
\end{enumerate}
If $n$ is even then we say that $M$ and $N$ are \emph{evenly linked}.
Analogously we can define the notion of  CI-linkage class, quasi-linkage class, even CI-linkage class and even quasi-linkage class (of modules).

\section{A Construction}
In this section we describe a construction essentially due to Ferrand. Throughout 
$(A,\m)$ is  a Gorenstein local ring of dimension $d$.

\s  \label{dual} We note the following well-known result, see \cite[3.3.10]{BH}. 
  Let $D \in \CMg(A)$. Then $\Ext^i_A(D, A) = 0 $ for $i \neq g$. Set $D^\dagger = \Ext^g_A(D, A)$. Then
$D^\dagger \in \CMg(A)$. Furthermore $(D^\dagger)^\dagger \cong D$.

The following result is well-known. However we give a proof as we do not have a reference.
\begin{lemma}\label{qgor-mod}
Let $(A,\m)$ be  a Gorenstein local ring of dimension $d$ and let $M \in \CMg(A)$. Let $\q$ be a quasi-Gorenstein ideal of grade $g$ contained in $\ann M$. Set $B = A/\q$. Then
$\Ext^g_A(M, A) \cong \Hom_B(M, B)$.
\end{lemma}
\begin{proof}
Let $\by = y_1,\ldots, y_g \in \q$  be a regular sequence. Set $C = A/(\by)$. Then
we have a natural ring homomorphism $C \rt B$. As $C,B$ are Gorenstein rings we have
$\Hom_C(B,C) \cong B$, see \cite[3.3.7]{BH}. We now note that
\begin{align*}
\Ext^g_A(M, A) &\cong \Hom_C(M, C), \ \ \ \ \ \ \text{see \cite[3.1.16]{BH}} \\
&= \Hom_C(M\otimes_B B, C), \\
&\cong \Hom_B(M, \Hom_C(B,C)), \\
&= \Hom_B(M, B).
\end{align*}
\end{proof}
\s \label{const} \emph{Construction:}   Let $M \in \CMg(A)$ and let $\q$ be a quasi-Gorenstein ideal in $A$ of codimension $g$ contained in $\ann M$. Let $M \sim_\q N$. Let $\Pb$ be minimal free resolution of $M$ and let $\Qb$ be a minimal free resolution of $\Pb_0/\q \Pb_0$.
We have a natural map $ \Pb_0/\q \Pb_0 \rt M \rt 0$. We lift this to a chain map
$\phi \colon \Qb \rt \Pb$. We then dualize this map to get a chain map $\phi^* \colon \Pb^* \rt \Qb^*$. Let $\Cb = \cone(\phi^*)$.

\begin{lemma}\label{coh}
\[
H^i(\Cb) = \begin{cases}
N, &\text{if $i = g$,} \\
0, &\text{if $i \neq g$}.
\end{cases}
\]
\end{lemma}
\begin{proof}
We have an exact sequence of complexes
\[
0 \rt \Qb^* \rt \Cb \rt \Pb^*(-1) \rt 0.
\]
Notice
\[
H^i(\Pb^*) = \Ext^i_A(M, A) = \begin{cases}
M^\dagger = \Ext^g_A(M, A), &\text{if $i = g$,} \\
0, &\text{if $i \neq g$}.
\end{cases}
\]
We also have
\[
H^i(\Qb^*) = \Ext^i_A(\Pb_0/\q \Pb_0, A) = 
0 \ \text{if $i \neq g$}.
\]
It is now immediate that
\[
H^i(\Cb) = 0  \quad \text{for} \ i \neq g-1, g.
\]
Set $B = A/\q$ and $\ov{P} = \Pb_0/\q \Pb_0$. Then by \ref{qgor-mod} we have 
\begin{align*}
\Ext^g_A(M, A) &\cong \Hom_B(M, B), \text{and} \\ 
\Ext^g_A(\ov{P}, A) &\cong \Hom_B(\ov{P}, B).
\end{align*}
We have an exact sequence of MCM $B$-modules
\[
0 \rt K \rt \ov{P} \xrightarrow{\epsilon} M \rt 0.
\]
Dualizing \wrt \ $B$ we get an exact sequence
\[
0 \rt \Hom_B(M, B)  \xrightarrow{\epsilon^*} \Hom_B(\ov{P}, B) \rt K^* \rt 0.
\]
As $M \sim_\q N$ we get that $N \cong K^*$. 
We note that the map $H^g(\Pb^*) \rt H^g(\Qb^*)$ is $\epsilon ^*$. As a consequence we obtain that
\[
H^{g-1}(\Cb) = 0 \quad \text{and} \quad H^g(\Cb) = N.
\]
\end{proof}
 \begin{remark}
 \begin{enumerate}
 \item 
 If $A$ is regular local, $M = A/I$,  $\q$ is a complete intersection and $I \sim_\q J$, then this construction was used by Ferrand to give a projective resolution of $A/J$, see \cite[Proposition 2.6]{PS}.
\item 
 If $M$ is perfect $A$-module of codimension $g$ and $\q$ is a Gorenstein ideal then 
 this construction was used by  Martsinkovsky and Strooker to give a projective resolution of $N$, see \cite[Proposition 10, p.\ 597]{MS}.
 \end{enumerate}
 \end{remark}

Our interest in this construction is due to the following:
\s \label{obs-2}\emph{Observation:} Let $B^g(\Cb)$ be the module of $g$-boundaries of $\Cb$ and
$Z^g(\Cb)$ be the module of $g$-cocycles of $\Cb$. Then $\projdim_A  B^g(\Cb)$ is finite and $Z^g(\Cb)$ is a maximal \CM \ $A$-module. Thus the sequence
\[
0 \rt B^g(\Cb) \rt Z^g(\Cb) \rt N \rt 0,
\]
is a maximal \CM \ approximation of $N$.\\
\emph{Proof} This follows from Lemma \ref{coh}.

\section{Proof of Theorem \ref{third}(1)}
In this section we prove Theorem \ref{third}(1). It is an easy consequence of the following result:

\begin{theorem}\label{rachel}
Let $(A,\m)$ be a Gorenstein local ring of dimension $d$. Let $M\in \CMg(A)$.
Assume $M \sim_\q N$ where
  $\q$ is  a Gorenstien ideal in $A$. Let $L$ be a maximal \CM \ $A$-module.
   Let $\Pb$ be minimal free resolution of $N$ and let $\Qb$ be a minimal free resolution of $\Pb_0/\q \Pb_0$. We do the construction as in \ref{const} and let $\Cb = \cone(\phi^*)$. Set $X = Z^g(\Cb)$ and $Y = B^g(\Cb)$.  For $s \geq 1$, let $X_s$ be the image of the map $\Pb^*_{s-1} \rt \Pb^*_{s}$. Let $\bx = x_1,\ldots,x_d$ be a maximal regular $A$-sequence. Set $\ov{A} = A/(\bx)$ and $\ov{L} = L/\bx L$. Then the following assertions are equivalent:
\begin{enumerate}[\rm(i)]
\item
$\Ext^i_A(L, M) = 0$ for all $i \gg 0$.
\item
$\Ext^i_A(L, X) = 0$ for all $i \gg 0$.
\item
For all $s \gg 0$ and for all $i \geq 1$ we have $\Ext^i_A(L, X_s) = 0$.
\item
$H^i(\Hom_A(L, \Pb^*)) = 0$ for all $i \gg 0$.
\item
$H^i(\Hom_A(L\otimes_A \Pb, A)) = 0$ for all $i \gg 0$.
\item
$H^i(\Hom_{\ov{A}}(\ov{L}\otimes_A \Pb, \ov{A})) = 0$ for all $i \gg 0$.
\item
$\Tor^A_i(\ov{L}, N) = 0$ for $i \gg 0$.
\item
$\Tor^A_i(L, N) = 0$ for $i \gg 0$.
\end{enumerate}
\end{theorem}

We need a few preliminary results before we are able to prove Theorem \ref{rachel}.

\s \label{trunc} Let $\K : \cdots \rt K_n \xrightarrow{d_n} K_{n+1} \rt \cdots $ be a co-chain complex. Let $s$ be an integer. By  $\K_{\geq s}$ we mean the co-chain complex
\[
0\rt K_s \xrightarrow{d_s} K_{s+1} \rt \cdots \rt K_n \xrightarrow{d_n} K_{n+1} \rt \cdots 
\]
Clearly $H^i(\K_{\geq s}) = H^i(\K)$ for all $i \geq s+1$.

\s \label{trunc-cone} Let $\Pb, \Qb, \Cb$ be as in Theorem \ref{rachel}. As $\q$ is a Gorenstein ideal, it has in particular finite projective dimension. It follows that for $i \geq d + 1$, we get
$\Cb_i = \Pb_{i+1}^*$ and the map $\Cb_{i} \rt \Cb_{i+1}$ is same as the map $\Pb_{i+1}^* \rt \Pb_{i+2}^*$. In particular if $s \geq d + 1$ then
$\Cb_{\geq s} = \Pb^*_{\geq s+1}$. 

We need the following:
\begin{lemma}\label{mcm-complex}
Let $(A,\m)$ be a Gorenstein local ring and let $\Db$ be a chain-complex with $\Db_n = 0$ for $n \leq -1$. Assume that $\Db_n$ is a maximal \CM \ $A$-module for  all $n \geq 0$. Let $x\in \m$ be $A$-regular. Let  $\ov{A} = A/(x)$ and let $\ov{\Db}$ be the
the  complex $\Db\otimes_A \ov{A}$. Let $\Db^*$ be the complex $\Hom_{A}(\Db, A)$ and let $\ov{\Db}^*$ be the complex $\Hom_{\ov{A}}(\ov{\Db}, \ov{A})$. Then the following are equivalent:
\begin{enumerate}[\rm(i)]
\item
$H^i(\Db^*) = 0 $ for $i \gg 0$.
\item
$H^i(\ov{\Db}^*) = 0 $ for $i \gg 0$.
\end{enumerate}
\end{lemma}
\begin{proof}
Let $E$ be a  maximal \CM \ $A$-module. Notice $x$ is $E$-regular. Furthermore
$\Hom_A(E,A)$ is a maximal \CM \ $A$-module and $\Ext^i_A(M, A) = 0$ for $i \geq 1$. Set $\ov{E} = E/xE$.

The exact sequence $0 \rt A \xrightarrow{x} A \rt \ov{A} \rt 0$ induces an exact sequence
\[
0 \rt \Hom_A(E,A) \xrightarrow{x} \Hom_A(E,A) \rt \Hom_{\ov{A}}(\ov{E},\ov{A}) \rt 0.
\]
Thus we have an exact sequence of co-chain complexes of $A$-modules
\[
0 \rt \Db^*  \xrightarrow{x} \Db^*  \rt \ov{\Db}^* \rt 0.
\]
This in turn induces a long-exact sequence in cohomology 
\begin{equation}\label{les}
\cdots \rt H^i(\Db^*) \xrightarrow{x} H^i(\Db^*) \rt H^i(\ov{\Db}^*) \rt H^{i+1} (\Db^*) \rt \cdots
\end{equation}

We now prove
$(i) \implies (ii)$. This follows from (\ref{les}).

$(ii) \implies (i)$. From (\ref{les}) we get that for all $i \gg 0$ the map
$H^i(\Db^*) \xrightarrow{x} H^i(\Db^*)$ is surjective. The result now follows from Nakayama's Lemma.
\end{proof}
As an easy consequence of \ref{mcm-complex} we get the following:
\begin{corollary}\label{mcm-complex-cor}
(with same hypotheses as in Lemma \ref{mcm-complex}). Let $\bx = x_1,\ldots,x_d$ be a maximal regular sequence in $A$. Set $B = A/(\bx)$, $\K = \Db\otimes_A B$,
Let $\K^*$ be the complex $\Hom_{B}(\K, B)$. Then the following are equivalent:
\begin{enumerate}[\rm(i)]
\item
$H^i(\Db^*) = 0 $ for $i \gg 0$.
\item
$H^i(\K^*) = 0 $ for $i \gg 0$.
\end{enumerate}
\end{corollary}

We now give
\begin{proof}[Proof of Theorem \ref{rachel}]
$(i) \iff (ii)$. By \ref{obs-2} we get that $X$ is a maximal \CM \ $A$-module, $\projdim_A Y $ is finite and we have an exact sequence
\[
0 \rt Y \rt X \rt M \rt 0.
\]
As $A$ is Gorenstein we have that $\injdim_A Y$ is finite. It follows that
$\Ext^i_A(L, X) \cong \Ext^i_A(L, M)$ for all $i \geq d+1$. The result follows.

$(ii) \iff (iii)$. For $s \geq 1 $ let $\widetilde{X_s} = \image(\Cb_s \rt \Cb_{s+1})$.  For $s \geq g +1 $ by Lemma \ref{coh} we get that $\widetilde{X_s}$ is maximal \CM.   By Lemma \ref{coh} we also have an exact sequence
\[
0 \rt X \rt \Cb_g \rt \Cb_{g + 1} \rt \cdots \rt \Cb_s \rt \widetilde{X_s} \rt 0.
\]
It follows that if $s \gg 0$ then $\Ext^i_A(L,X) = 0 $ for $i \gg 0$ is equivalent to \\  $\Ext^i_A(L,\widetilde{X_s}) = 0 $ for $i \geq 1$. The result now follows as $\q$ is a Gorenstein ideal; so we have $\widetilde{X_s} = X_{s+1}$ for $s \geq d +1$, see \ref{trunc-cone}.

$(iii) \implies (iv)$. Suppose $\Ext^i_A(L,X_s) = 0 $ for all $i \geq 1 $ and all $s \geq c$. Set $  a = \max\{ g +1, c\}$.
As $H^i(\Pb^*) = \Ext^A_i(N, A) = 0$ for $i \geq g +1 $ we have an exact sequence
\[
0 \rt X_a \rt \Pb^*_{a+1} \rt \Pb^*_{a+2} \rt \cdots   \rt \Pb^*_{n} \rt   \Pb^*_{n+1} \rt \cdots 
\]
As $\Ext^i_A(L,X_s) = 0$ for all $i \geq 1$ and all $s\geq a$ we get that the induced sequence
\[
0 \rt \Hom_A(L, X_a) \rt \Hom_A(L, \Pb^*_{a+1}) \rt \cdots
\]
is exact.
It follows that $H^i(L, \Pb^*_{\geq a + 1}) = 0 $ for $ i \geq a +2$.
Thus by \ref{trunc} we get that $H^i(L, \Pb^*) = 0$ for $ i \geq a + 2$.

$(iv) \implies (iii)$. Suppose $H^i(L, \Pb^*) = 0$ for $ i \geq r$. If $s \geq r$ then $H^i(L, \Pb^*_{\geq s}) = 0$ for all $i \geq s + 1$. Now let $a = \max\{ g +1, r\}$. Let $s \geq a$. As argued before we have an exact sequence
\[
0 \rt X_s \rt \Pb^*_{s+1} \rt \Pb^*_{s+2} \rt \cdots   \rt \Pb^*_{n} \rt   \Pb^*_{n+1} \rt \cdots 
\]
As $H^i(L, \Pb^*_{\geq s}) = 0$ for all $i \geq s + 1$ we get that
the induced sequence
\[
0 \rt \Hom_A(L, X_s) \rt \Hom_A(L, \Pb^*_{s+1}) \rt \cdots
\]
is exact. In particular $\Ext^A_1(L,X_s) = 0$. Notice $\Ext^2_A(L,X_s) \cong \Ext^1_A(L, X_{s+1})$. The latter module is zero by the same argument. Iterating we get $\Ext^i_A(L, X_s) = 0$ for all $i \geq 1$.

$(iv) \iff (v)$. We have an isomorphism of complexes
\[
\Hom_A(L, \Pb^*) \cong \Hom_A(L \otimes_A \Pb, A).
\]
The result follows.

$(v) \iff (vi)$. $\Db = L \otimes_A \Pb$ is a complex of maximal \CM \ $A$-modules since $L$ is maximal \CM \ and $\Pb_n$ is a finitely generated free $A$-module. Also
$\Db_n = 0$ for $n \leq -1$. Set $\K = \Db\otimes_A \ov{A} = \ov{L} \otimes_A \Pb$. The result now follows from Corollary \ref{mcm-complex-cor}.

$(vi) \iff (vii)$. We note that $\ov{A}$ is a zero-dimensional Gorenstein local ring and so is injective as an $\ov{A}$-module. Furthermore $\ov{A}$ is the injective hull of its residue field.
It follows that
\[
H^i(\Hom_{\ov{A}}(\ov{L}\otimes_A \Pb, \ov{A}))  \cong \Hom_{\ov{A}}(H_i(\ov{L}\otimes_A \Pb ), \ov{A}) = \Hom_{\ov{A}}(\Tor^A_i(\ov{L}, N), \ov{A}).
\]
Therefore by \cite[3.2.12]{BH} we get the result.

$(vii) \iff (viii)$. As $L$ is a maximal \CM \ $A$-module we get that $\bx$ is an $L$-regular sequence. Set $L_j = L/(x_1,\ldots,x_j)$ for $j = 1,\ldots,d$. Note $L_d = \ov{L}$.

 We have an exact sequence
$ 0 \rt L \xrightarrow{x_1} L \rt L_1 \rt 0$. This induces a long exact sequence
\[
\cdots \rt \Tor^A_i(L,N)  \xrightarrow{x_1} \Tor^A_i(L,N) \rt \Tor^A_i(L_1,N) \rt \Tor^A_{i-1}(L,N) \rt \cdots
\]
Clearly if $\Tor^A_i(L,N) = 0$ for $i \gg 0$ then $\Tor^A_i(L_1,N) = 0$ for all $i \gg 0$. Conversely if $\Tor^A_i(L_1,N) = 0$ for all $i \gg 0$ then for the map
$\Tor^A_i(L,N)  \xrightarrow{x_1} \Tor^A_i(L,N)$ is surjective for all $i \gg 0$. By Nakayama's Lemma we get $\Tor^A_i(L,N) = 0$ for $i \gg 0$.

We also have an exact sequence $ 0 \rt L_1 \xrightarrow{x_2} L_1 \rt L_2 \rt 0$. 
A similar argument gives that $\Tor^A_i(L_1,N) = 0$ for $i \gg 0$ if and only if  $\Tor^A_i(L_2,N) = 0$ for all $i \gg 0$. Combining with the previous result we get that $\Tor^A_i(L,N) = 0$ for $i \gg 0$ if and only if  $\Tor^A_i(L_2,N) = 0$ for all $i \gg 0$.

Iterating this argument yields the result.
\end{proof}

We now give
\begin{proof}[Proof of Theorem \ref{third}(1)]
Let $0 \rt W \rt E \rt L \rt 0$ be a maximal \CM \ approximation of $L$. As $\projdim_A W$ is finite we get that $\Ext^i_A(L,M) \cong \Ext^i_A(E,M)$ for $i \geq d + 2$ and $\Tor^A_i(L, N) \cong \Tor^A_i(E, N)$ for $i \geq d + 2$. The result now follows from Theorem \ref{rachel}.
\end{proof}

\section{Proof of Theorem \ref{third}(2) and Corollary \ref{third-cor}}
In this section we prove Theorem \ref{third}(2). We need a few preliminary facts to prove this result.

\s Let $(A,\m)$ be a Noetherian local ring and let $E$ be the injective hull of its residue field. If $G$ is an $A$-module then set $G^\vee = \Hom_A(G, E)$. Let $\ell(G)$ denote the length of $G$. The following result is known. We give a proof as we are unable to find a reference. 

\begin{lemma}\label{ext-tor}
Let $(A,\m)$ be a Noetherian local ring and let $E$ be the injective hull of its residue field. Let $M$ be  finitely generated $A$-module. Let $L$ be an $A$-module with  $\ell(L) < \infty$. Then for all $i \geq 0$ we have an isomorphism
\[
\Ext^i_A(M, L) \cong (\Tor^A_i(M,L^\vee))^\vee.
\]
\end{lemma}
\begin{proof}
We note that $(L^\vee)^\vee \cong L$, see \cite[3.2.12]{BH}.
Let $\Pb$ be a minimal projective resolution of $M$.
We have the following isomorphism of complexes:
\[
\Hom_A(\Pb\otimes_A L^\vee, E) \cong \Hom_A(\Pb, \Hom_A(L^\vee,  E)) \cong \Hom_A(\Pb,L).
\]
We have $H^i(\Hom_A(\Pb,L)) = \Ext^A_i(M, L)$.
Notice
as $E$ is an injective $A$-module we have
\begin{align*}
H^i(\Hom_A(\Pb\otimes_A L^\vee, E)) &\cong \Hom_A(H_i(\Pb\otimes_A L^\vee), E) \\
&\cong \Hom_A( \Tor^A_i(M, L^\vee), E)\\
&= (\Tor^A_i(M,L^\vee))^\vee.
\end{align*}
\end{proof}

The following result is also known. We give a proof as we are unable to find a reference.

\begin{lemma}\label{mod-x-dual}
Let $(A,\m)$ be a Gorenstein local ring of dimension $d$ and let $E$ be the injective hull of its residue field.  Let $D \in \CMr(A)$ and let $\bx = x_1,\ldots,x_{c} \in \m$ be a  $D$-regular sequence (note $c \leq d-r$). Then
\begin{enumerate}[\rm (1)]
\item
$\bx$ is a $D^\dagger$-regular sequence.
\item
$(D/\bx D)^\dagger \cong D^\dagger/\bx D^\dagger$.
\item
If $r = d$ (and so $\ell(D) < \infty$) then
\[
D^\dagger \cong D^\vee ( = \Hom_A(D, E)).
\]
\item
If $T$ is another finitely generated $D$-module then
\begin{enumerate}[\rm (a)]
\item
$\Ext^i_A(T, D) = 0$ for $i \gg 0$ if and only if $\Ext^i_A(T, D/\bx D) = 0$ for $i \gg 0$.
\item
$\Ext^i_A(D,T) = 0$ for $i \gg 0$ if and only if $\Ext^i_A( D/\bx D , T) = 0$ for $i \gg 0$.
\end{enumerate}
\end{enumerate}
\end{lemma}
\begin{proof}
(1) and (2) :
Let $x$ be $D$-regular. Then notice $D/xD$ is a \CM \ $A$-module of codimension $r +1$. We have a short-exact sequence $0 \rt D \xrightarrow{x} D \rt D/xD \rt 0$. Applying the functor $\Hom_A(-, A)$ yields a long exact sequence which after applying \ref{dual} reduces to a short-exact sequence
\[
0 \rt \Ext^r_A(D, A) \xrightarrow{x} \Ext^r_A(D, A) \rt \Ext^{r+1}_A(D/xD, A) \rt 0.
\]
It follows that $x$ is $D^\dagger$-regular and $D^\dagger/x D^\dagger \cong (D/xD)^\dagger$. 

Iterating this argument yields (1) and (2).

(3): Let $\by = y_1,\ldots, y_d \subset \ann_A D$ be a maximal $A$-regular sequence.
Set $B = A/(\by)$.
Then 
\[
D^\dagger = \Ext^n_A(D, A) \cong \Hom_B(D,B).
\]
We have
\[
D^\vee = \Hom_A(D, E) = \Hom_A(D\otimes_A B , E) \cong \Hom_B(M, \Hom_A(B, E)).
\]
Now $\Hom_A(B, E) $ is an injective $B$-module of finite length. As $B$ is an Artin Gorenstein local ring we get that  $\Hom_A(B, E)$ is free as a $B$-module.   As
$\ell(\Hom_A(B, E)) = \ell(B)$ (see \cite[3.2.12]{BH}) it follows that  $\Hom_A(B, E) = B$. The result follows. 

(4): For $i = 1,\ldots,c$ set $D_i = D/(x_1,\ldots,x_i)D$.

4(a): The exact sequence $0 \rt D \xrightarrow{x_1} D \rt D_1 \rt 0$ induces a long exact sequence 
\begin{equation}\label{md}
\cdots \Ext^i_A(T, D) \xrightarrow{x_1} \Ext^i_A(T, D) \rt \Ext^i_A(T, D_1) \rt \Ext^{i+1}_A(T, D) \rt \cdots
\end{equation}
If $\Ext^i(T, D) = 0$ for all $i \gg 0$ then by (\ref{md}) we get $\Ext^i_A(T, D_1) = 0$ for all $i \gg 0$. Conversely if $\Ext^i_A(T, D_1) = 0$ for all $i \gg 0$ then
by (\ref{md}) we get that the map $\Ext^i_A(T, D) \xrightarrow{x_1} \Ext^i_A(T, D)$ is surjective for all $i \gg 0$. So by Nakayama's Lemma $\Ext^i(T, D) = 0$ for all $i \gg 0$.

We also have an   exact sequence $0 \rt D_1 \xrightarrow{x_2} D_1 \rt D_2 \rt 0$.
A similar argument to the above yields that  $\Ext^i(T, D_1) = 0$ for all $i \gg 0$ 
if and only if $\Ext^i(T, D_2) = 0$ for all $i \gg 0$. Combining this with the previous result we get $\Ext^i(T, D) = 0$ for all $i \gg 0$ 
if and only if $\Ext^i(T, D_2) = 0$ for all $i \gg 0$. 

Iterating this argument yields the result.

4(b): This is similar to 4(a).
\end{proof}
We now give
\begin{proof}[Proof of Theorem \ref{third}(2)]
Let $\bx = x_1,\ldots,x_{d-r}$ be a maximal $D$-regular sequence. Then by \ref{mod-x-dual}, $\bx$ is also a $D^\dagger$ regular sequence. Also by \ref{mod-x-dual}, we get $D^\dagger/\bx D^\dagger = (D/\bx D)^\dagger$. Let $E$ be the injective hull of the residue field of $A$. 

We have
\begin{align*}
&\Ext^i_A(M, D)  = 0 \ \text{for all $i \gg 0$} \\
&\iff \Ext^i_A(M, D/\bx D)  = 0 \ \text{for all $i \gg 0$}; \text{see \ref{mod-x-dual}}, \\
&\iff \Tor^A_i(M, (D/\bx D)^\vee) = 0 \ \text{for all $i \gg 0$}; \text{see \ref{ext-tor} and \cite[3.2.12]{BH}},\\
 &\iff \Tor^A_i(M,D^\dagger/\bx D^\dagger) = 0 \ \text{for all $i \gg 0$}; \text{see \ref{mod-x-dual}}, \\
&\iff  \Ext^A_i(D^\dagger/\bx D^\dagger, N) = 0 \ \text{for all $i \gg 0$}; \text{see Theorem \ref{third}(1)}, \\
&\iff  \Ext^A_i(D^\dagger, N) = 0 \ \text{for all $i \gg 0$}; \text{see \ref{mod-x-dual}}.
\end{align*}
\end{proof}

We now give 
\begin{proof}[Proof of Corollary \ref{third-cor}]
By Theorem \ref{third}(2) we get that 
\[
\Ext^i_A(M, C)  = 0 \  \text{for all $i \gg 0$} \iff \Ext^i_A(C^\dagger, N) = 0 \  \text{for all $i \gg 0$}.
\]
We note that $C^\dagger = \Ext^r_A(C,A) \cong \Hom_{A/\n}(C, A/\n)$, see \ref{qgor-mod}.
As $C \sim_\n D$ we have an exact sequence $ 0 \rt C^\dagger \rt G \rt D \rt 0$, where $G$ is a finitely generated free $A/\n$-module. As $\n$ is a Gorenstein ideal in $A$ we get $\projdim_A G$ is finite. The result follows.
\end{proof}
\section{Some preliminaries to prove Theorem \ref{second}}
In this section we discuss a few preliminaries which will enable us to prove Theorem  \ref{second}. More precisely we need the notion of cohomological operators over a complete intersection ring; see \cite{Gull} and
\cite{Eis}.

\s 
 Let $\mathbf{f} = f_1,\ldots,f_c$ be a regular sequence in a  local Noetherian ring $(Q,\n)$. We assume $ \mathbf{f} \subseteq \n^2$.  Set $I = (\mathbf{f})$ and
 $ A = Q/I$.
\s

The \emph{Eisenbud operators}, \cite{Eis}  are constructed as follows: \\
Let $\mathbb{F} \colon \cdots \rightarrow F_{i+2} \xrightarrow{\partial} F_{i+1} \xrightarrow{\partial} F_i \rightarrow \cdots$ be a complex of free
$A$-modules.

\emph{Step 1:} Choose a sequence of free $Q$-modules $\wt{F}_i$ and maps $\wt{\partial}$ between them:
\[
\wt{\mathbb{F}} \colon \cdots \rightarrow \wt{F}_{i+2} \xrightarrow{\wt{\partial}} \wt{F}_{i+1} \xrightarrow{\wt{\partial}} \wt{F}_i \rightarrow \cdots
\]
so that $\mathbb{F} = A\otimes\wt{\mathbb{F}}$

\emph{Step 2:} Since $\wt{\partial}^2 \equiv 0 \ \text{modulo} \ (\mathbf{f})$, we may write  $\wt{\partial}^2  = \sum_{j= 1}^{c} f_j\wt{t}_j$ where
$\wt{t_j} \colon \wt{F}_i \rightarrow \wt{F}_{i-2}$ are linear maps for every $i$.

 \emph{Step 3:}
Define, for $j = 1,\ldots,c$ the map $t_j = t_j(Q, \mathbf{f},\mathbb{F}) \colon \mathbb{F} \rightarrow \mathbb{F}(-2)$ by $t_j = A\otimes\wt{t}_j$.

\s
The operators $t_1,\ldots,t_c$ are called Eisenbud's operator's (associated to $\mathbf{f}$) .  It can be shown that
\begin{enumerate}
\item
$t_i$ are uniquely determined up to homotopy.
\item
$t_i, t_j$ commute up to homotopy.
\end{enumerate}

\s Let $R = A[t_1,\ldots,t_c]$ be a polynomial ring over $A$ with variables $t_1,\ldots,t_c$ of degree $2$. Let $M, N$ be  finitely generated $A$-modules. By considering a free resolution $\mathbb{F}$ of $M$ we get well defined maps
\[
t_j \colon \Ext^{n}_{A}(M,N) \rightarrow \Ext^{n+2}_{R}(M,N) \quad \ \text{for} \ 1 \leq j \leq c  \ \text{and all} \  n,
\]
which turn $\Ext_A^*(M,N) = \bigoplus_{i \geq 0} \Ext^i_A(M,N)$ into a module over $R$. Furthermore these structure depend  on $\mathbf{f}$, are natural in both module arguments and commute with the connecting maps induced by short exact sequences.

\s  Gulliksen, \cite[3.1]{Gull},  proved that if $\projdim_Q M$ is finite then
$\Ext_A^*(M,N) $ is a finitely generated $R$-module. For $N = k$, the residue field of $A$, Avramov in \cite[3.10]{LLAV} proved a converse; i.e., if
$\Ext_A^*(M,k)$ is a finitely generated $R$-module then $\projdim_Q M$ is finite.

\s We need to recall the notion of complexity of a module.
 This notion was introduced by Avramov in \cite{LLAV}.
 Let $\beta_i^A(M) = \ell( \Tor^A_i(M,k) )$ be the $i^{th}$ Betti number of $M$ over $A$. The complexity of $M$ over $A$ is defined by
\[
\cx_A M = \inf\left \lbrace b \in \mathbb{N}  \left\vert \right.   \varlimsup_{n \to \infty} \frac{\beta^A_n(M)}{n^{b-1}}  < \infty \right \rbrace.
\]
If $A$ is a local complete intersection of $\codim c$ then $\cx_A M \leq c$. Furthermore all values between $0$ and $c$ occur.

\s \label{cdim} Since $\m \subseteq \ann \Ext^{i}_A(M,k)$ for all $i \geq 0$ we get that $\Ext^*_A(M,k)$ is a module over $S = R/\m R = k[t_1,\ldots,t_c]$.
If $\projdim_Q M$ is finite then $\Ext^*_A(M,k)$ is a finitely generated $S$-module of Krull dimension $\cx M$.

\s \label{cx-dual} If $(Q,\n)$ is regular then by \cite[Theorem I(3)]{ABu} we get that $\dim_S \Ext^*_A(M,k) = \dim_S  \Ext^*_A(k,M)$. In particular if $M$ is maximal \CM \ $A$-module then $\cx M = \cx M^*$. Using this fact it is not difficult to show that if $M$ is a \CM \ $A$-module then $\cx M = \cx M^{\dagger}$. 

\s \label{lucho} Let $Q$ is regular local with infinite residue field and let $M$ be a finitely generated $A$-module with $\cx(M) = r$. The surjection
$Q \rt A$ factors as $Q \rt  R  \rt A$, with the kernels of both maps generated by regular
sequences, $ \projdim_R M < \infty$ and $\cx_A M = \projdim_R A$ (see \cite[3.9]{LLAV}).

We need the following:
\begin{proposition}\label{comp-cx}
Let $Q = k[x,y,z]_{(x,y,z)}$ where $k$ is an infinite field. Let $\m$ be the maximal  
 ideal of $Q$. Let $\q$ be an $\m$-primary Gorenstein ideal such that $\q$ is not a 
 complete intersection. Suppose $\q \supseteq (u,v)$ where $u,v \in \m^2$ is an $Q$-regular sequence. Set $A = Q/(u,v)$ and $\ov{\q} = \q/(u,v)$. Then $\cx A/\ov{\q} = 2$.
\end{proposition} 
\begin{proof}
As $\codim A = 2$ we have that $\cx A/ \ov{\q} = 0 , 1$ or $2$. We prove $\cx A/ \ov{\q} \neq 0 , 1$.

If $\cx A/ \ov{\q} = 0$ then $A/\ov{\q}$ has finite projective dimension over $A$. So by Auslander-Buchsbaum formula $\projdim_A A/\ov{\q} = 1$. Therefore $\ov{\q}$ is a principal ideal. It follows that $\q$ is a complete intersection, a contradiction.

 If $\cx A/ \ov{\q} = 1$ then by \ref{lucho}, the surjection
$Q \rt A$ factors as $Q \rt  R  \rt A$, with the kernels of both maps generated by regular
sequences, $ \projdim_R A/\ov{\q} < \infty$ and $\cx_A A/\ov{\q} = \projdim_R A = 1$. Thus $\dim R = 2$.  So $R = Q/(h)$ for some $h$.

As $A/\ov{\q}$ has finite length, by Auslander-Buchsbaum formula $\projdim_R A/\ov{\q} = 2$. Consider the minimal resolution of  $A/\ov{\q}$ over $R$:
\[
0 \rt R^b \rt R^a \rt R \rt  A/\ov{\q} \rt 0.
\]
As $A/\ov{\q}$ is a Gorenstein ring we have $b = 1$ and so $a = 2$. 
Thus there exists $\alpha, \beta \in R$ with $A/\ov{\q} = R/(\alpha, \beta) = Q/(h,\alpha, \beta)$. It follows that $\q$ is a complete intersection ideal, a contradiction. 
\end{proof}
\section{Proof of Theorem \ref{second}}
In this section we prove Theorem \ref{second}. We first make the following:

\s\label{second-const}  \textit{Construction:} Let $Q = \mathbb{Q}[x,y,z]_{(x,y,z)}$  and let $\m$ be its maximal ideal. We construct a $\m$-primary Gorenstein ideal $\q$ in $Q$ such that
\begin{enumerate}
\item
$\q$ is not a complete intersection.
\item
There  exists a $Q$-regular sequence $f,u,v \in \m^2$ such that
$$ (f^a, v^b, u) \subseteq \q \subseteq (f,u,v) \quad \text{for some} \ a,b \geq 2. $$
\end{enumerate}
Set $u = x^2 + y^2 + z^2$, $R = Q/(u), \n = \m/(u)$. Then note $(x,y)$ is a reduction of $\n$ and $\n^2 = (x,y)\n$. It follows that $x^7,y^7$ is a regular sequence in $R$. So $I = (x^7, y^7) \colon (xy + yz + xz)$ is a Gorenstein ideal. Using Singular, \cite{DGPS},  it can be shown that $I$ has $12$ minimal generators and $I \subseteq \n^6$. In particular $I$ is not a complete intersection in $R$.
Also note that $\n^6 \subseteq (x, y)^5 \subseteq (x^2, y^2)$. Let $\q$ be an ideal in $Q$ containing $u$ such that $\q/(u) = I$. Then $\q$ has $12$ or $13$ minimal generators. So $\q$ is not a complete intersection. Also clearly $\q$ is $\m$-primary. Note
\[
((x^2)^4, (y^2)^4, u) \subseteq (x^7, y^7, u) \subseteq \q \subseteq (x^2, y^2, u).
\]
Set $f = x^2, v = y^2$ and $a = b = 4$.

We now give 
\begin{proof}[Proof of Theorem \ref{second}]
Let $Q = \mathbb{Q}[x,y,z]_{(x,y,z)}$  and let $\m$ be its maximal ideal.
We make the construction as in \ref{second-const}.
Let $E$ be a non-free stable maximal \CM \ $Q/(f)$-module. Let 
\[
0 \rt Q^r \rt Q^r \rt E \rt 0
\]
be a minimal resolution of $E$ as a $Q$-module. Note $u$ is $Q/(f)$-regular and so 
$E$-regular. Thus we have an exact sequence
\[
0 \rt (Q/(u))^r \rt (Q/(u))^r \rt E/uE \rt 0.
\]

Set $A = Q/(u,f^a)$ and $\ov{\q} = \q/(u, f^a)$. Then note that $\ov{\q}$ is a quasi-Gorenstein ideal in $A$ which is not a complete intersection. Furthermore notice that $E/uE$ is a maximal \CM \  $Q/(f,u)$-module and so a maximal \CM \ $A$-module.
Furthermore notice that $\cx E/uE = 1$ as an $A$-module.
 
We now note that $v$ is $A$-regular and so $E/uE$-regular. Set $M = E/(u,v)E$. Thus $M$ is an $A$-module of finite length. Also $\cx_A M = 1$. Furthermore 
\[
\q \subseteq (f,u,v) \subseteq \ann_Q M.
\]
So we have $\ov{\q} \subseteq \ann_A M$.  Thus we have finite length module $M$ of complexity one and a quasi-Gorenstein ideal $\ov{\q}$ of complexity two. Set $B = A/ \ov{\q}$. Clearly $M$ does not have $B$ as a direct summand. Set $N = \Omega_B(\Tr_B(M))$. Then $M$ is horizontally linked to $N$ as $B$-modules, see \cite[Proposition 8, p.\ 596]{MS}. So $M \sim_{\ov{\q}} N$ as $A$-modules. 
If $t \in \ann_A M$ is a regular element then it is not difficult to show that there exists $i \geq 1$ such that the $C= A/(t^i)$-module $M$ has no free summands as a $C$-module. Let $M \sim_{t^i} L$ as $A$-modules.

By Lemma \ref{gor-cx} we have: $\cx_A N = 2$ and $\cx_A L = 1$.

This finishes the proof as for any module $P$ the complexity of a maximal \CM \ approximation of $P$ is equal to complexity of $P$. We note that $N$ is evenly linked to $L$ and $\cx X(N) = \cx N = 2$ while $\cx X(L) = \cx L = 1$. It follows that $X(N)$ is not stably isomorphic to $X(L)$.

\end{proof}
We now state and prove the  Lemma we need to finish the proof of Theorem \ref{second}.

\begin{lemma}\label{gor-cx}
Let $(Q,\n)$ be a regular local ring and let $\bff = f_1,\ldots, f_c \in \n^2$ be a regular sequence. Set $A = Q/(\bff)$. 
Let $M \in \CMg(A)$.  Also let  $M \sim_\q N$ where $\q$ is a quasi-Gorenstein ideal in $A$. Then
\begin{enumerate}[\rm (1)]
\item
If $\q$ is a Gorenstein ideal then $\cx M = \cx N$.
\item
If $\projdim A/\q = \infty$ and $\cx A/\q > \cx M$ then
$ \cx N = \cx A/\q$. 
\end{enumerate}
\end{lemma}
\begin{proof}
Set $B = A/\q$. Let $M^\dagger = \Ext^g_{A}(M,A) \cong \Hom_B(M, B)$, by Lemma \ref{qgor-mod}. As $M \sim_\q N$ we have an exact sequence
\begin{equation}\label{gor-cx-eqn}
0 \rt M^\dagger \rt G \rt N \rt 0; \quad \text{where} \ G \ \text{is free $B$-module}.
\end{equation}
By \cite[3.3]{ABu} we get $\cx M^\dagger = \cx M$. The result now follows from the exact sequence \ref{gor-cx-eqn} and \ref{cdim}. 
\end{proof}
\section{Proof of Theorem \ref{first}}
In this section we prove Theorem \ref{first}. We need to prove several preliminary results first.

\s Let $M, N$ be finitely generated $A$-modules. By $\beta(M,N)$ we mean the subset of $\Hom_A(M,N)$ which factor through a finitely generated free $A$-module. We first prove the following  
\begin{proposition}
\label{basic}
Let $(A,\m)$ be a Gorenstein local ring. Let $M$ be a maximal \CM \ $A$-module with no free summands. Then $\beta(M,M) \subseteq \rad \End(M)$.
\end{proposition}
\begin{proof}
Let $f \in \beta(M,M)$. Say $f = v\circ u$ where $u \colon M \rt F$, $v \colon F \rt M$  and $F  = A^n$. Let $u = (u_1,\ldots,u_n)$ where $u_i \colon M \rt A$. As $M$ does not have a free summand we get that $u_i(M) \subseteq \m$ for each $i$. Thus $u(M) \subseteq \m F$. It follows that $f(M) \subseteq \m M$. Thus $f \in \Hom_A(M, \m M)$. However it is well-known and easy to prove that $\Hom_A(M, \m M)  \subseteq \rad \End(M)$.
\end{proof}

\s \label{set-up} It can be easily seen that $\beta(M,M)$ is a two-sided ideal in $\End(M)$.  Set $\Endu(M) = \End(M)/\beta(M,M)$.  Assume $A$ is Gorenstein and  $M$ is maximal \CM \ with no free summands. Let 
\[
0 \rt N \xrightarrow{u} F \xrightarrow{\pi} M \rt 0,
\]
be a minimal presentation. Note that $N$ is also maximal \CM \ with no free-summnads.
We construct a ring homomorphism
\[
\sigma \colon \Endu(M) \rt \Endu(N)
\]
as follows:
Let $\theta \in \End(M)$. Let $\delta \in \End(N)$ be a lift of $\theta$. We first note that if $\delta^\prime $ is another lift of $\theta$ then it can be easily shown that  there exists $\xi \colon F \rt N$ such that $\xi  \circ u = \delta - \delta^\prime$. Thus we have a well defined element $\sigma(\theta) \in \Endu(N)$. Thus we have a map
\[
\widetilde{\sigma} \colon \End(M) \rt \Endu(N)
\]
It is easy to see that $\widetilde{\sigma}$ is a ring homomorphism. We prove
\begin{proposition}\label{mod}
(with hypotheses as above) If $ f \in \beta(M, M)$ then $\widetilde{\sigma}(f) = 0$. 
\end{proposition}
\begin{proof}
Let $\mathbb{F}$ be a minimal resolution of $M$ with $\mathbb{F}_0 = F$. Suppose $f = \psi \circ \phi$ where $\psi \colon G \rt M$, $\phi \colon M \rt G$ and $G = A^m$ for some $m$. We may take $\mathbb{G}$ be a minimal resolution of $G$ with $\mathbb{G}_0 = G$ and $\mathbb{G}_n = 0$ for $n > 0$. We can construct a lift $\widetilde{f} \colon \mathbb{F} \rt \mathbb{F}$  of $f$ by composing a lift of $v$ with that of $u$. It follows that $\widetilde{f}_n = 0$ for $n \geq 1$. An easy computation shows that for this lift $\widetilde{f}$ the corresponding map $\delta \colon N \rt N$ is infact \textit{zero}. Thus $\widetilde{\sigma}(f) = 0$.
\end{proof}
Thus we have a ring homomorphism $\sigma \colon \Endu(M) \rt \Endu(N)$.
Our next result is
\begin{proposition}\label{iso}
(with hypotheses as above) $\sigma$ is an isomorphism.
\end{proposition}
\begin{proof}
We construct a ring homomorphism $\tau \colon \Endu(N) \rt \Endu(M)$ which we show is the inverse of $\sigma$. 

Let $\theta \colon N \rt N$ be $A$-linear. Then $\theta^* \colon  N^* \rt N^*$ is also $A$-linear. We dualize the exact sequence $0 \rt N \rt F \rt M \rt 0$ to get an exact sequence
\[
0 \rt M^* \xrightarrow{\pi^*} F^* \xrightarrow{u^*} N^* \rt 0.
\]
We can lift $\theta^*$ to an $A$-linear map $\delta \colon M^* \rt M^*$. Also if $\delta^\prime$ is another lift then as before it is easy to see $\delta - \delta^\prime \in 
\beta(M^*,M^*)$.  We define $\widetilde{\tau} \colon \End(N) \rt \Endu(M)$ by
$\widetilde{\tau}(\theta) = \delta^*$. It is clear that $\widetilde{\tau}$ is a ring homomorphism. Finally as  in Proposition \ref{mod}, it can be easily proved that if 
$\theta \in \beta(N, N)$ then $\widetilde{\tau}(\theta) = 0$. Thus we have a ring homomorphism $\tau \colon \Endu(M) \rt \Endu(N)$. Finally it is tautalogical that
$$\tau \circ \sigma = 1_{\Endu(M)} \quad \text{and} \quad \sigma \circ \tau = 1_{\Endu(N)}. $$ 
\end{proof}
\s \label{obs} If $\phi \colon R \rt S $ is an isomophism of rings then it is easy to see that $\phi( \rad R) = \rad S$ and thus we have an isomorphism of rings $R/\rad R \cong S/\rad S$. 
As a corollary we obtain
\begin{corollary}
\label{bc}
(with hypotheses as above)
$$ \End(M)/\rad \End(M) \cong \End(N)/\rad \End(N).$$
\end{corollary}
\begin{proof}
By Proposition \ref{iso} we have an isomorphism $\sigma \colon \Endu(M) \rt \Endu(N)$.
By \ref{basic} we have that $\beta(M,M) \subseteq \rad \End(M)$.  It follows that  
$$\rad \Endu(M) = \rad \End(M)/\beta(M,M).$$
Similarly $\rad \Endu(N) = \rad \End(N)/\beta(N,N)$. The result follows from \ref{obs}.
\end{proof}

We now give 
\begin{proof}[Proof of Theorem \ref{first}]
It suffices to consider to prove that if 
\[
M_0 \sim_{\q} M_{1}
\]
then $\End_A(M_0)/\rad \End_A(M_0) \cong (\End_A(M_0)/\rad \End_A(M_0))^{op}$. By assumption $M_0, M_1 \in \CMg(A)$. It follows that $\q$ is a codimension $g$ quasi-Gorenstein ideal \cite[Lemma 14, p.\ 616]{MS}. Set $B = A/\q$.  Then $B$ is a Gorenstein ring. Notice $M_0, M_1$ are maximal \CM \ $B$-modules. Furthermore they are stable $B$-modules, see \cite[Proposition 3, p.\ 593]{MS}.
Notice $\Hom_A(M_i, M_i) = \Hom_B(M_i, M_i)$ for $i = 0, 1$. 

As $M_0$ is horizontally linked to $M_1$ we get that $M_0^* \cong \Omega(M_1)$. By
\ref{bc} we get that
\[
\End(M_1)/\rad \End(M_1) \cong \End(M_0^*)/\rad \End(M_0^*).
\]
Furthermore it is easy to see that $\End_B(M_0) \cong \End_B(M_0^*)^{op}$ and this is preserved when we go mod radicals. Thus
\[
\End(M_1)/\rad \End(M_1) \cong \left(\End(M_0)/\rad \End(M_0)\right)^{op}.
\]
\end{proof}

\end{document}